\documentclass[english,11pt]{amsart}

\usepackage{graphicx}
\usepackage{amsmath}
\usepackage{amssymb}
\usepackage{amsfonts}
\usepackage{amsopn}
\usepackage{amsthm}
\usepackage{amscd}
\usepackage[all]{xy}
\usepackage[latin1]{inputenc}
\usepackage[T1]{fontenc}
\usepackage[french]{babel}
\usepackage{latexsym}
\usepackage{hyperref} 

\renewcommand{\proofname}{Proof}

\theoremstyle{plain}
\newtheorem{thm}{Theorem}
\newtheorem{lem}[thm]{Lemma}
\newtheorem{cor}[thm]{Corollary}
\newtheorem{prop}[thm]{Proposition}

\theoremstyle{definition}

\newtheorem{ex}[thm]{Example}
\newtheorem{eg}[thm]{Example}

\newtheorem{exo}[thm]{Exercise}
\newtheorem{rem}[thm]{Remarque}
\newtheorem{defi}[thm]{Definition}
\newtheorem{defn}[thm]{Definition}

\newtheorem{nota}[thm]{Notation}

\newcommand{\bgt}{\begin{thm}}
\newcommand{\et}{\end{thm}}
\newcommand{\bgp}{\begin{prop}}
\newcommand{\ep}{\end{prop}}
\newcommand{\bgd}{\begin{defi}}
\newcommand{\ed}{\end{defi}}
\newcommand{\bgl}{\begin{lem}}
\newcommand{\el}{\end{lem}}
\newcommand{\bgr}{\begin{rem}}
\newcommand{\er}{\end{rem}}
\newcommand{\bge}{\begin{ex}}
\newcommand{\ee}{\end{ex}}
\newcommand{\bgex}{\begin{exo}}
\newcommand{\eex}{\end{exo}}
\newcommand{\bgi}{\begin{itemize}}
\newcommand{\ei}{\end{itemize}}
\newcommand{\bgpe}{\begin{proof}}
\newcommand{\epe}{\end{proof}}
\newcommand{\bgar}{\begin{array}}
\newcommand{\ear}{\end{array}}
\newcommand{\bgn}{\begin{nota}}
\newcommand{\en}{\end{nota}}
\newcommand{\bgc}{\begin{cor}}
\newcommand{\ec}{\end{cor}}
\newcommand{\bgen}{\begin{enumerate}}
\newcommand{\een}{\end{enumerate}}

\newcommand{\qform}[1]{{\langle{#1}\rangle}}  

\DeclareMathOperator{\Aut}{{Aut}}

\DeclareMathOperator{\ind}{{ind}}  

\DeclareMathOperator{\res}{res}

\newcommand{\Sp}{\mathrm{Sp}}

\newcommand{\Spin}{\mathrm {Spin}}

\DeclareMathOperator{\Sl}{SL}

\DeclareMathOperator{\su}{SU}

\newcommand{\la}{\lambda}
\newcommand{\ot}{\otimes}

\newcommand{\Gt}{\widetilde{G}}
\newcommand{\ksep}{k_{\mathrm{sep}}}
\newcommand{\Gm}{\mathbb{G}_m}
\newcommand{\Z}{\mathbb{Z}}
\newcommand{\F}{\mathbb{F}}
\newcommand{\QQ}{\mathbb{Q}}
\newcommand{\RR}{\mathbb{R}}

\DeclareMathOperator{\car}{char}
\DeclareMathOperator{\Dec}{Dec}
\DeclareMathOperator{\Gal}{Gal}

\newcommand{\darkrad}{0.17}
\newcommand{\lrad}{0.4}
\newcommand{\rb}[1]{\raisebox{0.1in}[0pt]{#1}}

\newcommand{\barr}{\bar{r}}

\newcommand{\pform}[1]{{\langle\!\langle{#1}\rangle\!\rangle}} 

\title{On the Tits $p$-indexes of semisimple algebraic groups}

\author{Charles De Clercq}

\author{Skip Garibaldi}
\thanks{SG's research was partially supported by NSF grant DMS-1201542.  Part of this research was performed while SG was at the Institute for Pure and Applied Mathematics (IPAM), which is supported by the National Science Foundation.}

\date{\today}


\begin{document}
\renewcommand{\proofname}{Proof}
\renewcommand{\refname}{References}
\renewcommand{\abstractname}{Abstract}


\bigskip
\maketitle

\begin{abstract}
The first author has recently shown that semisimple algebraic groups are classified up to motivic equivalence by the local versions of the classical Tits indexes over field extensions, known as Tits p-indexes.  We provide in this article the complete description of the values of the Tits p-indexes over fields. From this exhaustive study, we also deduce criteria of motivic equivalence for semisimple groups of many types, hence giving a dictionary between classic algebraic structures, representation theory, cohomological invariants and Chow motives of the twisted flag varieties for those groups.
\end{abstract}

The \emph{Tits index} (sometimes called Satake diagram) of a semisimple linear algebraic group $G$ over a field $k$ includes as special cases the classical notions of Schur index of a central simple associative algebra and the Witt index of a quadratic form.  It is a fundamental invariant of semisimple algebraic groups.
However, for the purpose of stating and proving theorems about Chow motives with $\mathbb{F}_p$ coefficients, one should consider not the Tits index of $G$, but rather the (Tits) $p$-index, meaning the Tits index of $G_L$ where $L$ is an algebraic extension of $k$ of degree not divisible by $p$, yet all the finite algebraic extensions of $L$ have degree a power of $p$.  Such an $L$ is called a \emph{$p$-special closure} of $k$ in \cite[\S101.B]{EKM} and all such fields are isomorphic as $k$-algebras, so the notion of Tits $p$-index over $k$ is well defined.

Let $G$ be a semisimple algebraic group over $k$. As shown in \cite{clercq:motivicequivalence},  the Tits $p$-indexes of $G$ on all fields extensions of $k$ --- the \emph{higher Tits $p$-indexes of $G$} --- determine the motivic equivalence class of $G$ modulo $p$. The aim of this article is to determine the values of the Tits $p$-indexes of the absolutely simple algebraic groups, using as a starting point the known list of possible Tits indexes as in \cite{Ti:Cl}, \cite{Sp:LAG}, or \cite{PetrovStavrova}. Along the way, we give in some cases  criteria for motivic equivalence for semisimple groups in terms of their algebraic and cohomological invariants.

\section{Generalities}

\subsection{Definition of the Tits index \cite{Ti:Cl}, \cite{PetrovStavrova}, \cite[\S1]{MPW1}}
The \emph{Tits index} is (1) the Dynkin diagram of $G$, which we conflate with its set $\Delta$ of vertices, together with (2) the action of the absolute Galois group $\mathrm{Gal}(k)$ of $k$ on $\Delta$, and (3) a $\mathrm{Gal}(k)$-invariant subset $\Delta_0\subset \Delta$.  Specifically, pick a maximal $k$-torus $T$ in $G$ containing a maximal $k$-split torus $S$.  For $\ksep$ a separable closure of $k$ --- so $\Gal(k)$ are the $k$-automorphisms of $\ksep$ --- $T \times \ksep$ and $G \times \ksep$ are split, and, from the set $\Phi$ of roots of $G \times \ksep$ with respect to $T \times \ksep$ one picks a set of simple roots $\Delta$.  As $T$ is $k$-defined, $\Gal(k)$ acts naturally on $\Phi$; this action need not preserve $\Delta$, but modifying it by elements of the Weyl group in a natural way gives a canonical action of $\Gal(k)$ on $\Delta$, called the \emph{$*$-action}.  The resulting graph (Dynkin diagram with vertex set $\Delta$) with action by $\Gal(k)$ is uniquely defined up to isomorphism in the category of graphs with a $\Gal(k)$-action; it does not depend on the choice of $T$ or $\Delta$.

Choose orderings on $T^* \ot_\Z \RR$ and $S^* \ot_\Z \RR$ such that the linear map, restriction $T^* \to S^*$, takes nonnegative elements of $T^*$ to nonnegative elements of $S^*$.  Define $\Delta_0$ to be the set of $\alpha \in \Delta$ such that $\alpha\vert_S = 0$; it is a union of $\Gal(k)$-orbits in $\Delta$.
One has $\Delta_0=\Delta$ iff $G$ is anisotropic and $\Delta_0=\emptyset$ iff $G$ is quasi-split. The elements of $\delta_0=\Delta \setminus \Delta_0$ are called \emph{distinguished} and the number of $\Gal(k)$-orbits of distinguished elements equals the rank of a maximal $k$-split torus in $G$.

To represent the Tits index graphically, one draws the Dynkin diagram and circles the distinguished vertices.  Traditionally, one indicates the Galois action by drawing vertices in the same $\Gal(k)$-orbit physically close to each other on the page, and by using one large circle or oval to enclose each $\Gal(k)$-orbit in $\delta_0$.  The Tits index of $G$ has no circles iff $G$ is anisotropic, and every vertex is circled iff $G$ is quasi-split.

The definition of Tits index is compatible with base change, as explained carefully in \cite[pp.~115, 116]{PrRap:fields}. That is, for each extension $E$ of $k$, the Tits index of $G \times E$ may be taken to have the same underlying graph (the Dynkin diagram with vertex set $\Delta$) with $\Gal(E)$-action given by the restriction map $\Gal(E) \to \Gal(k)$, and with set of distinguished vertices containing the distinguished vertices in the Tits index of $G$.

\subsection{Which primes $p$?}
If every simple group of a given type is split by a separable extension of degree not divisible by $p$, then the only possible $p$-index is the split one.  Thus, \cite[\S2.2]{SeCGp} gives a complete list $S(G)$ of the primes meriting consideration, which we reproduce in Table \ref{SG.table}.  (Or see \cite{Ti:deg} for more precise information on degrees of splitting fields.)

To say this in a different way, we fix a prime $p$ and will describe the possible Tits indexes of a simple algebraic group over a field $k$ that is \emph{$p$-special}, i.e., such that every finite extension has degree a power of $p$. 

\begin{table}[hbt]
\begin{tabular}{c|cc} \\
type of $G$&exponent of the center&elements of $S(G)$ \\ \hline
$A_n$&$n+1$&2 and the prime divisors of $n+1$\\
$B_n$, $C_n$, $D_n$ ($n \ne 4$)&2&2 \\
$G_2$&1&2 \\
$D_4$, $E_7$&2&2 and 3\\
$F_4$&1&2 and 3\\
$E_6$&3&2 and 3\\
$E_8$&1&2, 3, and 5
\end{tabular}
\caption{Primes $S(G)$ as in \cite{SeCGp}} \label{SG.table}
\end{table}

\subsection{Twisted flag varieties and motivic equivalence \cite{BoTi:C}, \cite{BoSp2}, \cite[\S1]{MPW1}}
For each subset $\Theta$ of $\Delta$, there is a parabolic subgroup $P_\Theta$ of $G \times \ksep$ (determined by the choice of simple roots) whose Levi subgroup has Dynkin diagram $\Delta \setminus \Theta$.  The (projective) quotient variety $(G \times \ksep)/P_\Theta$ is the variety of parabolic subgroups of $G \times \ksep$ that are conjugate to $P_\Theta$.  This variety is $k$-defined if and only if $\Theta$ is invariant under $\Gal(k)$, in which case we denote it by $X_\Theta$.  These varieties are the \emph{twisted flag varieties} of $G$.  

Now fix a $\Gal(k)$-invariant subset $\Theta$ of $\Delta$.  The following are equivalent: \emph{(1) $X_\Theta$ has a $k$-point; (2) $X_\Theta$ is a rational variety; (3) $\Theta \subseteq \delta_0$.}  In this way, the Tits index of $G$ gives information about $X_\Theta$.  

In case $X_\Theta$ does not have a $k$-point, the Chow motive of $X_\Theta$ nonetheless gives information about the geometry of $X_\Theta$.
Suppose now that $G'$ is a semisimple group over $k$, and that the quasi-split inner forms 
$G_{\nu_G}$, $G'_{\nu_{G'}}$ of $G$, $G'$ are isogenous.  That is, suppose that there is an isomorphism $f$ from the Dynkin diagram $\Delta$ of $G$ to that of $G'$ that commutes with the action of $\Gal(k)$.  This defines a correspondence $X_\Theta \leftrightarrow X_{f(\Theta)}$ between the twisted flag varieties of $G$ and $G'$.  The groups $G$, $G'$ are \emph{motivic equivalent modulo a prime $p$} if there is a choice of $f$ such that the mod-$p$ Chow motives of $X_\Theta$ and $X_{f(\Theta)}$ are isomorphic for every $\Gal(k)$-stable $\Theta \subseteq \Delta$.  The groups $G$, $G'$ are \emph{motivic equivalent} if the groups are motivic equivalent mod $p$ for every prime $p$ (where $f$ may depend on $p$).  
The main result of \cite{clercq:motivicequivalence} says, for a prime $p$: \emph{$G$ and $G'$ are motivic equivalent mod $p$ iff there is an isomorphism $f$ whose base change to each $p$-special field $E$ containing $k$ identifies the distinguished vertices in the Tits index of $G \times E$ with those in the Tits index of $G' \times E$.}  Informally, $G$ and $G'$ are motivic equivalent mod $p$ iff $G \times K$ and $G' \times K$ have the same Tits $p$-index for every extension $K$ of $k$.  This theorem is one motivation for our study of the possible Tits $p$-indexes of semisimple groups.

\subsection{The quasi-split type of $G$} \label{qs.type}
A group $G$ over $k$ has \emph{quasi-split type} $^tT_n$ if, upon base change to a separable closure $\ksep$ of $k$, $G \times \ksep$ is split with root system of type $T_n$ and if the image of $\Gal(k) \to \Aut(\Delta)$ has order $t$.  If $t = 1$, then $G$ is said to have \emph{inner type} and otherwise $G$ has \emph{outer type}.  
In the case where $k$ is $p$-special, evidently $t$ must be a power of $p$.

\subsection{The Tits class of $G$}
Suppose $G$ is adjoint with simply connected cover $\Gt$; one has an exact sequence $1 \to Z \to \Gt \to G \to 1$ where $Z$ is the scheme-theoretic center of $\Gt$.  This gives a connecting homomorphism $\partial \!: H^1(k, G) \to H^2(k, Z)$; here and below $H^i$ denotes fppf cohomology.  There is a unique class $\nu_G \in H^1(k, G)$ such that twisting $G$ by $\nu_G$ gives a quasi-split group \cite[31.6]{KMRT}, and we call $t_G := -\partial(\nu_G) \in H^2(k, Z)$ the \emph{Tits class} of $G$.

As a finite abelian group scheme, there is a unique minimal natural number $n$ such that multiplication by $n$ is the zero map on $Z$, it is called the exponent of $Z$.  If $p$ does not divide $n$ and $k$ is $p$-special, then the Tits class of $G$ is necessarily zero.

\begin{lem} \label{tits.class}
Suppose $G$ is a semisimple adjoint algebraic group with simply connected cover $\Gt$.
If $t_G = 0$, then there is a unique class $\xi_G \in H^1(k, \Gt)$ so that the twisted group $\Gt_{\xi_G}$ is quasi-split.
\end{lem}

\begin{proof}
As $t_G = 0$, the exactness of the sequence $H^1(k, \Gt) \to H^1(k, G) \xrightarrow{\partial} H^2(k, Z)$ shows that there is a $\xi \in H^1(k, \Gt)$ mapping to $\nu_G$, and it remains to prove uniqueness.  

By twisting, it is the same to show that, for $\Gt$ quasi-split simply connected, the map $H^1(k, \Gt) \to H^1(k, \Aut(\Gt))$ has zero kernel.  By \cite[Theorem 11, Example 15]{G:outer}, the kernel of the map is the image of $H^1(k, Z) \to H^1(k, \Gt)$.  As $Z$ is contained in a quasi-trivial maximal torus $S$ of $\Gt$,  the map factors through $H^1(k, S) = 0$. 
\end{proof}

\subsection{The Tits algebras of $G$ \cite[\S27]{KMRT}, \cite{Ti:R}}
Let $\Gt$, $Z$ be as in the previous subsection.
The Tits class provides, by way of the Tits algebras, a cohomological obstruction for an irreducible representation of $\Gt \times \ksep$ over $\ksep$ to be defined over $k$.  Specifically, such a representation has highest weight a dominant weight $\la$.  Put $k(\la)$ for the subfield of $\ksep$ of elements fixed by the stabilizer of $\la$ in $\Gal(k)$  (under the $*$-action).  The weight $\la$ is fixed by $\Gal(k(\la))$, i.e., $\la$ restricts to a homomorphism $Z \to \Gm$. The \emph{Tits algebra} $A_G(\la)$ is the image of $t_G$ under the induced map $\la \!: H^2(k(\la), Z) \to H^2(k(\la), \Gm)$; the irreducible representation is defined over $k(\la)$ iff $A_G(\la) = 0$.  (Note that in this definition $A_G(\la)$ is only a Brauer class, but a more careful definition gives a central simple algebra whose degree equals the dimension of the representation.)

As $H^2(k, Z)$ is a torsion abelian group where every element has order dividing the exponent of $Z$, there is a unique element $t_{G,p}$ of $p$-primary order such that $t_G - t_{G,p}$ has order not divisible by $p$.  We write $A_{G,p}(\la) \in H^2(k(\la), \Gm)$ for the image of $t_{G,p}$ under $\la$; it is the $p$-primary component of the Brauer class $A_G(\la)$.  We now show that if $G$ and $G'$ are motivic equivalent mod $p$, then they must have the same Tits algebras ``up to prime-to-$p$ extensions''.  Note that if $G$ and $G'$ are motivic equivalent mod $p$ for some $p$, then the isomorphism $f$ provides an identification of the centers of their simply connected covers, and we may write $Z$ for both centers.

\begin{prop}
Suppose $G$ and $G'$ are absolutely simple algebraic groups that are motivic equivalent mod $p$ via an isomorphism of Dynkin diagrams $f$.  Then $t_{G,p}$ and $t_{G',p}$ generate the same subgroup of $H^2(k, Z)$ and, for every dominant weight $\la$, $A_{G,p}(\la)$ and $A_{G',p}(f(\la))$ generate the same subgroup of $H^2(k(\la), \Gm)$.
\end{prop}

\begin{proof}
For sake of contradiction, suppose that $G$ and $G'$ are motivic equivalent mod $p$ with respect to some isomorphism of their Dynkin diagrams $f$, yet for some $\la$ that $A_{G,p}(\la)$ and $A_{G',p}(f(\la))$ generate distinct subgroups of $H^2(k(\la), Z)$ -- in particular the groups have types $A$, $B$, $C$, $D$, $E_6$, or $E_7$.

First consider the case where $k(\la) = k$ and $G$ has inner type but not $D_n$ for $n$ even.  Then $Z^*$ is cyclic and for each minuscule fundamental weight $\la_0$ that generates $Z^*$, $A_G(\la)$ is a multiple of $A_G(\la_0)$ and similarly for $G'$, so we may assume that $\la = \la_0$.  By hypothesis, and swapping the roles of $G$ and $G'$ if necessary, the subgroup $A_{G,p}(\la)$ does not contain $A_{G',p}(\la)$.  Put $\alpha$ for the simple root dual to $\la$ and set $K$ to be a $p$-special closure of the function field of the twisted flag variety $X_\alpha$ for $G$.  Then $\alpha$ is distinguished in the Tits index of $G \times K$ (trivially), while $A_{G',p}(f(\la)) \ot K$ is not split by \cite[Th.~B]{MT} so the simple root $f(\alpha)$ is not distinguished for $G' \times K$ by \cite[p.~211]{Ti:R}, providing the desired contradiction.  

If $k(\la) = k$ and $G$ has inner type $D_n$ for $n$ even, then $p = 2$, and by a similar argument we may replace the given $\la$ by one of the minuscule fundamental weights.

Suppose $k(\la) = k$ and $G$ has outer type.  If $G$ has type $^2A_n$, then as $\la\vert_Z$ is not zero, $n$ is odd and $\la$ is a sum of the unique $\Gal(k)$-fixed fundamental weight $\la_0$ and a weight that restricts to zero on $Z$, hence $A_G(\la) = A_G(\la_0)$ and we may replace $\la$ with $\la_0$ and argue as in the previous cases.  A similar argument treats the case of type $^2D_n$ for $n$ odd.  For type $^2D_n$ with $n$ even, we may replace $\la$ with the maximal weight of the natural module.  The case of types $^3D_4$, $^6D_4$, and $^2E_6$ do not occur, as every $\Gal(k)$-stable dominant weight restricts to zero on $Z$.

If $k(\la)$ is a proper extension of $k$, then replacing $G$, $G'$, $k$ by $G \times k(\la)$, $G' \times k(\la)$, $k(\la)$ produces absolutely simple groups over $k(\la)$ with the same Tits algebras at $\la$.  Applying the $k = k(\la)$ case again gives a contradiction.

Finally, if $t_{G,p}$ and $t_{G',p}$ generate different subgroups of $H^2(k, Z)$, then as the map $H^2(k, Z) \to \prod_{\chi \in Z^*} H^2(k(\chi), Z)$ is injective \cite[Prop.~7]{G:outer}, there is some $\la$ such that $A_{G,p}(\la)$ and $A_{G',p}(\la)$ generate different subgroups of $H^2(k(\la), Z)$.
\end{proof}

\subsection{The Rost invariant \cite{GMS}, \cite{KMRT}}
We refer to \cite[pp.~105--158]{GMS} for the precise definition of the abelian torsion groups $H^3(k, \Z/d\Z(2)) \to H^3(k, \QQ/\Z(2))$; if $d$ is not divisible by $\car k$, then $H^3(k, \Z/d\Z(2)) = H^3(k, \mu_d^{\otimes 2})$ and for all $k$, the natural inclusion identifies $H^3(k, \Z/d\Z(2))$ with the $d$-torsion in $H^3(k, \QQ/\Z(2))$.
For $\Gt$ a simple simply connected algebraic group, there is a canonical morphism of functors 
\[
r_{\Gt} \!: H^1(*, \Gt) \to H^3(*, \QQ/\Z(2))
\]
known as the \emph{Rost invariant}.  
The order $n_{\Gt}$ of $r_{\Gt}$ is known as the Dynkin index of $\Gt$, and $r_{\Gt}$ can be viewed as a morphism $H^1(*, \Gt) \to H^3(*, \Z/n_{\Gt}\Z(2))$.

\begin{lem} \label{rost.m}
Let $\Gt$ be absolutely simple and simply connected with center $Z$.  Put $m$ for the largest divisor of $n_{\Gt}$ that is relatively prime to the exponent of $Z$.  Then there is a canonical morphism of functors $\barr_{\Gt}$ such that 
the diagram
\[
\begin{CD}
H^1(*, \Gt/Z) @>{\barr_{\Gt}}>> H^3(*, \Z/m\Z(2)) \\ 
@AAA @AA{\pi}A \\
H^1(*, \Gt) @>{r_{\Gt}}>> H^3(*, \Z/n_{\Gt}\Z(2))
\end{CD}
\]
commutes, where $\pi$ is the projection arising from the Chinese remainder decomposition of $H^3(*, \Z/n_{\Gt}\Z(2))$.
\end{lem}    

Under the additional hypothesis that $m$ is not divisible by $\car k$, this result was proved in \cite[Prop.~7.2]{GaGi} using the elementary theory of cohomological invariants from \cite{G:lens}.  We give a proof valid for all characteristics that relies on the (deeper) theory of invariants of degree 3 of semisimple groups developed in \cite{M:ssinv2}.

\begin{proof}
Put $G := \Gt/Z$.  For $G$ of inner type, this result is included in the calculations in \cite[\S4]{M:ssinv2}.  
Consulting the list of Dynkin indexes for groups of outer type from \cite{GMS}, we have $m  = 1$ except for types $^2A_n$ with $n$ even, types $^3D_4$ or $^6D_4$, and type $^2E_6$, where $m$ is respectively $2$, $3$ and $4$.  To complete the proof, we calculate the group denoted
$Q(G)/\Dec(G)$ in \cite{M:ssinv2}.  Note that $Q(G)$ may be calculated over an algebraic closure, so the calculations in \cite{M:ssinv2} show that $Q(G)$ is $(n+1)\Z q$, $2\Z q$, or $3 \Z q$ respectively.
 
 The group $\Dec(G)$ is $n_G \Z q$ where $n_G$ is the gcd of the Dynkin index of each representation $\rho$ as $\rho$ varies over the $k$-defined representations of $G$.  Clearly, $n_G$ is unchanged by replacing $G$ by a twist by a 1-cocycle $\eta \in H^1(k, G)$, and $n_{\Gt_\eta}$ divides $n_{G_\eta}$.  Consulting then the maximum values for $n_{\Gt_\eta}$ from \cite{GMS}, we conclude that $n_G$ is divisible by $2(n+1)$, 12, 12 respectively.  On the other hand, $n_G$ divides the Dynkin index of the adjoint representation, which is twice the dual Coxeter number; hence $n_G$ divides $2(n+1)$, 12, 24 respectively.  Thus the maximal divisor of $|Q(G)/\Dec(G)|$ that is relatively prime to the exponent of $Z$ divides $m$ and the claim follows from the main theorem of \cite{M:ssinv2}.
 
For completeness, we note that for type $^2E_6$, the Weyl module with highest weight $\omega_1 + \omega_6$ of dimension 650 has Dynkin index 300 \cite{McKPR}, so $n_G$ divides $\gcd(24, 300) = 12$, i.e., $n_G = 12$.
\end{proof}

\begin{defn} \label{am.def}
Suppose $G$ is an absolutely almost simple algebraic group, and put $\Gt$, $\bar{G}$ for its simply connected cover and adjoint quotient.  For $m$ as defined in Lemma \ref{rost.m}, we define:
\[
b(G) := -\barr_{\Gt}(\nu_{\bar{G}}) \quad \in H^3(k, \Z/m\Z(2)).
\]
If $t_G = 0$, we define:
\[
a(G) := -r_{\Gt}(\xi_{\bar{G}}) \quad \in H^3(k, \Z/n_{\Gt}\Z(2))
\]
for $\xi_{\bar{G}}$ as in Lemma \ref{tits.class}. 

Factoring $n_{\Gt} = mc$, we have $H^3(k, \Z/n_{\Gt}\Z(2)) = H^3(k, \Z/m\Z(2)) \oplus H^3(k, \Z/c\Z(2))$.  In case $t_G = 0$, by Lemma \ref{rost.m} the  invariants are related by the equation $a(G) = b(G) + c(G)$ for some $c(G) \in H^3(k, \Z/c\Z(2))$.
\end{defn}

\begin{table}[bth]
{\centering\noindent\makebox[450pt]{
\begin{tabular}[c]{p{2.2in}|p{3in}}

$\small{(A_n)~~}$
\begin{picture}(7,2)(0,0)
\put(0,1){\circle*{3}}
\put(0,1){\line(1,0){20}}
\put(20,1){\circle*{3}}
\put(20,1){\line(1,0){20}}
\put(40,1){\circle*{3}}
\put(40,-1.6){ \mbox{$\cdots$}}
\put(62,1){\circle*{3}}
\put(62,1){\line(1,0){20}}
\put(82,1){\circle*{3}}
\put(82,1){\line(1,0){20}}
\put(102,1){\circle*{3}}

\put(-2,-7){\mbox{\tiny $1$}}
\put(18,-7){\mbox{\tiny $2$}}
\put(38,-7){\mbox{\tiny $3$}}
\put(54,-7){\mbox{\tiny $n$$-$$2$}}
\put(75,-7){\mbox{\tiny $n$$-$$1$}}
\put(100,-7){\mbox{\tiny $n$}}
\end{picture}
\vspace{0.5cm}

&

$\small{(E_6)~~}$
\begin{picture}(7,2)(0,0)
\put(0,-5){\circle*{3}}
\put(0,-5){\line(1,0){15}}
\put(15,-5){\circle*{3}}
\put(15,-5){\line(1,0){15}}
\put(30,-5){\circle*{3}}
\put(30,10){\circle*{3}}
\put(30,-5){\line(0,1){15}}
\put(30,-5){\line(1,0){15}}
\put(45,-5){\circle*{3}}
\put(45,-5){\line(1,0){15}}
\put(60,-5){\circle*{3}}

\put(-2,-13){\mbox{\tiny $1$}}
\put(13,-13){\mbox{\tiny $2$}}
\put(28,-13){\mbox{\tiny $3$}}
\put(43,-13){\mbox{\tiny $4$}}
\put(58.5,-13){\mbox{\tiny $5$}}
\put(33,9){\mbox{\tiny $6$}}
\end{picture}

\\

$\small{(B_n)~~}$
\begin{picture}(7,2)(0,0)
\put(0,1){\circle*{3}}
\put(0,1){\line(1,0){20}}
\put(20,1){\circle*{3}}
\put(20,1){\line(1,0){20}}
\put(40,1){\circle*{3}}
\put(40,-1.6){ \mbox{$\cdots$}}
\put(62,1){\circle*{3}}
\put(62,1){\line(1,0){20}}
\put(82,1){\circle*{3}}
\put(82,2){\line(1,0){20}}
\put(82,0){\line(1,0){20}}
\put(89,-1){{\tiny\mbox{$<$}}}
\put(102,1){\circle*{3}}

\put(-2,-7){\mbox{\tiny $1$}}
\put(18,-7){\mbox{\tiny $2$}}
\put(38,-7){\mbox{\tiny $3$}}
\put(54,-7){\mbox{\tiny $n$$-$$2$}}
\put(75,-7){\mbox{\tiny $n$$-$$1$}}
\put(100,-7){\mbox{\tiny $n$}}
\end{picture}
\vspace{0.5cm}

&

$\small{(E_7)~~}$
\begin{picture}(7,2)(0,0)
\put(0,-5){\circle*{3}}
\put(0,-5){\line(1,0){15}}
\put(15,-5){\circle*{3}}
\put(15,-5){\line(1,0){15}}
\put(30,-5){\circle*{3}}
\put(30,-5){\line(1,0){15}}
\put(45,-5){\circle*{3}}
\put(45,-5){\line(1,0){15}}
\put(45,10){\circle*{3}}
\put(45,-5){\line(0,1){15}}
\put(60,-5){\circle*{3}}
\put(60,-5){\line(1,0){15}}
\put(75,-5){\circle*{3}}

\put(-2,-13){\mbox{\tiny $1$}}
\put(13,-13){\mbox{\tiny $2$}}
\put(28,-13){\mbox{\tiny $3$}}
\put(43,-13){\mbox{\tiny $4$}}
\put(58.5,-13){\mbox{\tiny $5$}}
\put(73,-13){\mbox{\tiny $6$}}
\put(47.5,9){\mbox{\tiny $7$}}
\end{picture}

\\

$\small{(C_n)~~}$
\begin{picture}(7,2)(0,0)
\put(0,1){\circle*{3}}
\put(0,1){\line(1,0){20}}
\put(20,1){\circle*{3}}
\put(20,1){\line(1,0){20}}
\put(40,1){\circle*{3}}
\put(40,-1.6){ \mbox{$\cdots$}}
\put(62,1){\circle*{3}}
\put(62,1){\line(1,0){20}}
\put(82,1){\circle*{3}}
\put(82,2){\line(1,0){20}}
\put(82,0){\line(1,0){20}}
\put(89,-1){{\tiny\mbox{$<$}}}
\put(102,1){\circle*{3}}

\put(-2,-7){\mbox{\tiny $1$}}
\put(18,-7){\mbox{\tiny $2$}}
\put(38,-7){\mbox{\tiny $3$}}
\put(54,-7){\mbox{\tiny $n$$-$$2$}}
\put(75,-7){\mbox{\tiny $n$$-$$1$}}
\put(100,-7){\mbox{\tiny $n$}}
\end{picture}
\vspace{0.5cm}

&

$\small{(E_8)~~}$
\begin{picture}(7,2)(0,0)
\put(0,-5){\circle*{3}}
\put(0,-5){\line(1,0){15}}
\put(15,-5){\circle*{3}}
\put(15,-5){\line(1,0){15}}
\put(30,-5){\circle*{3}}
\put(30,-5){\line(1,0){15}}
\put(45,-5){\circle*{3}}
\put(60,-5){\line(0,1){15}}
\put(60,-5){\circle*{3}}
\put(60,10){\circle*{3}}
\put(75,-5){\circle*{3}}
\put(75,-5){\line(1,0){15}}
\put(90,-5){\circle*{3}}
\put(45,-5){\line(1,0){15}}
\put(60,-5){\line(1,0){15}}

\put(-2,-13){\mbox{\tiny $1$}}
\put(13,-13){\mbox{\tiny $2$}}
\put(28,-13){\mbox{\tiny $3$}}
\put(43,-13){\mbox{\tiny $4$}}
\put(58.5,-13){\mbox{\tiny $5$}}
\put(73,-13){\mbox{\tiny $6$}}
\put(88,-13){\mbox{\tiny $7$}}
\put(62.5,9){\mbox{\tiny $8$}}
\end{picture}

\\

$\small{(D_n)~~}$
\begin{picture}(7,2)(0,0)
\put(0,1){\circle*{3}}
\put(0,1){\line(1,0){20}}
\put(20,1){\circle*{3}}
\put(20,1){\line(1,0){20}}
\put(40,1){\circle*{3}}
\put(40,-1.6){ \mbox{$\cdots$}}
\put(62,1){\circle*{3}}
\put(62,1){\line(1,0){20}}
\put(82,1){\circle*{3}}
\put(82,2){\line(4,3){15}}
\put(82,0){\line(4,-3){15}}
\put(96.5,12.9){\circle*{3}}
\put(96.5,-10.9){\circle*{3}}

\put(-2,-7){\mbox{\tiny $1$}}
\put(18,-7){\mbox{\tiny $2$}}
\put(38,-7){\mbox{\tiny $3$}}
\put(54,-7){\mbox{\tiny $n$$-$$3$}}
\put(86,-0.5){\mbox{\tiny $n$$-$$2$}}
\put(100,-12){\mbox{\tiny $n$}}
\put(100,11.8){\mbox{\tiny $n$$-$$1$}}
\end{picture}

&

$\small{(F_4)~~}$
\begin{picture}(7,2)(0,0)
\put(0,1){\circle*{3}}
\put(0,1){\line(1,0){15}}
\put(15,1){\circle*{3}}
\put(15,1){\line(1,0){15}}
\put(30,1){\circle*{3}}
\put(30,1){\line(1,0){15}}
\put(45,1){\circle*{3}}

\put(-2,-7){\mbox{\tiny $1$}}
\put(13,-7){\mbox{\tiny $2$}}
\put(28,-7){\mbox{\tiny $3$}}
\put(43,-7){\mbox{\tiny $4$}}

\put(65,1){$\small{(G_2)~~}$}
\put(92,1){\circle*{3}}
\put(92,0.1){\line(1,0){15}}
\put(92,1.1){\line(1,0){15}}
\put(92,2.1){\line(1,0){15}}
\put(107,1){\circle*{3}}

\put(90,-7){\mbox{\tiny $1$}}
\put(105,-7){\mbox{\tiny $2$}}
\end{picture}

\end{tabular}
}}
\caption{Dynkin diagrams of simple root systems, with simple roots numbered}
\end{table}

\section{Tits $p$-indexes of classical groups} \label{classical.sec}

As envisioned by Siegel and Weil \cite{Weil}, classical groups can be described over a field $k$ of odd characteristic as automorphism groups of central simple algebras with involutions. Recall that a simple $k$-algebra is said to be central if its $k$-dimension is finite and if its center is $k$. The degree $\deg(A)$ of a central simple algebra is the square root of its dimension and its index $\ind(A)$ is the degree of a division $k$-algebra Brauer equivalent to $A$.

An involution $\sigma$ on a central simple algebra $A$ is a $k$-antiautomorphism of order $2$. Following \cite{KMRT}, the index $\ind(A,\sigma)$ of a central simple algebra is a set of the reduced dimensions of the $\sigma$-isotropic right ideals in $A$. As \cite[§6]{KMRT} states, $\ind(A,\sigma)$ is of the form $\{0,\ind(A),...,r\ind(A)\}$, where $r$ is the Witt index of an (skew-)hermitian form attached to $(A,\sigma)$.

\begin{defi}
Let $p$ be a prime and $(A,\sigma)$ be a central simple $k$-algebra with involution. The $p$-index of $(A,\sigma)$ is the union of the indexes of $A_E$, where $E$ runs through all coprime to $p$ field extensions of $k$.
\end{defi}

Entirely analogous statements hold for a central simple algebra with quadratic pair $(A, \sigma, f)$ as defined in \cite{KMRT}; one simply adapts the notion of isotropic right ideal and $\ind(A, \sigma, f)$ as in \cite[pp.~73, 74]{KMRT}.

In this section we will list the possible Tits $p$-indexes for the classical groups and relate them to the index of a corresponding central simple algebra with involution or quadratic pair.  The actual list of indexes is identical with the one in \cite{Ti:Cl}; the core new material here is the proof of existence of such indexes over $2$-special fields given in the next section.  We do not discuss motivic equivalence for classical groups here, because the cases of special orthogonal groups and groups of type $^1A_n$ were handled in \cite{clercq:motivicequivalence}.


\subsection{Type $^{1\!\!}A_n$}

An absolutely simple simply connected groups of type $^{1\!\!}A_n$ is isomorphic to the special linear group $\Sl_1(A)$ of a degree $n+1$ central simple $k$-algebra $A$. 
The coprime-to-$p$ components of $A$ vanishes over a $p$-special closure of $k$, hence denoting by $d_p$ the integer $p^{v_p(\ind(A))}$ we get the following description of the $p$-indexes of type $^{1\!\!}A_n$.\\

{\centering\noindent\makebox[360pt]{
  \includegraphics[width=14cm]{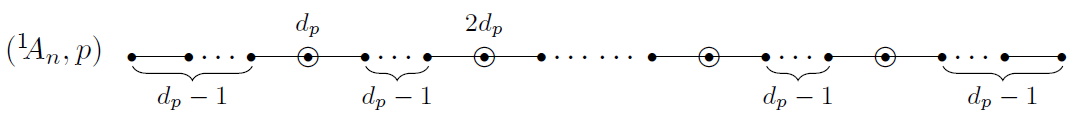}
}}

\vspace{0.5cm}
\emph{Distinguished orbits}: $\delta_0^p(\Sl_1(A))=\{d_p,2d_p,...,n+1-d_p\}$.\\

The twisted flag varieties of type $^{1\!}A_n$ are isomorphic to the varieties of flags of right ideals of fixed dimension \cite{MPW1}. Note that any power of $p$ may be realized as the integer $d_p$ for a central simple algebra defined over a suitable field $k$.

\subsection{Type $^{2\!\!}A_n$}

The absolutely simple simply connected groups of type $^{2\!\!}A_n$ correspond to the special unitary groups $\su(A,\sigma)$, where $(A,\sigma)$ is a central simple algebra of degree $n+1$ with involution of the second kind (recall that in this case $A$ is not central simple over $k$). As in \ref{qs.type}, we need only consider the prime 2.
We denote by $d$ the index of $A$ and $r$ is the integer such that $\ind_2(A,\sigma)=\{d_2,2d_2,...,rd_2\}$.\\

{\centering\noindent\makebox[360pt]{
  \includegraphics[width=14cm]{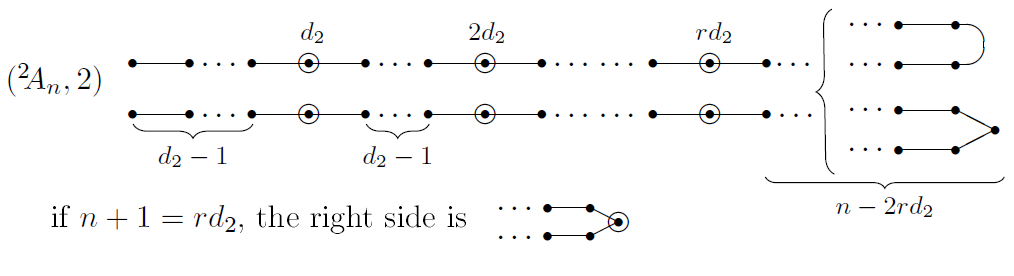}
}}

\emph{Distinguished orbits}: $\delta_0^2(G)=\{\{d_2,n+1-d_2\},\{2d_2,n+1-2d_2\},...,\{rd_2,n+1-rd_2\}\}$.\\

The associated twisted flag varieties are described in \cite{MPW1}. We will show in \S\ref{Tignol.sec} that any of such Tits $2$-index can be realized by a group of type $^{2\!}A_n$ defined over a suitable field.



\subsection{Type $B_n$}

An absolutely simple simply connected group of type $B_n$ is isomorphic to the spinor group $\Spin(V,q)$ of a $(2n+1)$-dimensional quadratic space $(V,q)$ (the adjoint groups of type $B_n$ correspond to the special orthogonal groups). The \emph{Witt index} of the quadratic space $(V,q)$ is denoted by $i_w(q)$ and the only torsion prime here is $2$.\\

{\centering\noindent\makebox[360pt]{
  \includegraphics[width=11cm]{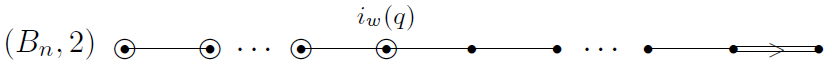}
}}




 










\emph{Distinguished orbits}: $\delta_0^2(G)=\{1,2,...,i_w(q)\}$.\\

The Tits $2$-index coincides with the classical Tits index for such groups by Springer's theorem (see \cite{Sp:conjecture}, \cite[Corollary 18.5]{EKM} for a characteristic-free proof). It is thus known from Tits classification that any of such Tits $2$-index can be achieved as the index of the spinor group of some quadratic space $(V,q)$.


\subsection{Type $C_n$}

Up to isomorphism, an absolutely simple simply connected group of type $C_n$ is the symplectic group $\Sp(A,\sigma)$ associated to a central simple $k$-algebra of degree $2n$ with symplectic involution. As previously the only torsion prime here is $2$ and we denote the $2$-index $\ind_2(A,\sigma)=\{d,2d,...,rd\}$, where $d$ is the index of $A$.\\

{\centering\noindent\makebox[360pt]{
  \includegraphics[width=12cm]{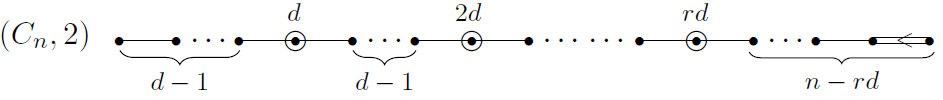}
}}

\begin{tabular}{ll}
if $n=rd$, the right side becomes&

\begin{picture}(7,2)

\put(0,-0.5){\mbox{$\cdots$}}

\put(17,2){\circle*{3}}
\put(17,2){\line(1,0){20}}
\put(37,2){\circle*{3}}
\put(57,2){\circle*{3}}
\put(57,2){\circle{7}}
\put(37,3){\line(1,0){20}}
\put(37,1){\line(1,0){20}}
\put(44,0.2){\mbox{\scriptsize{\textbf{$<$}}}}

\end{picture}
\end{tabular}

\vspace{0,8cm}

\emph{Distinguished orbits}: $\delta_0^2(G)=\{d,2d,...,rd\}$.\\

The twisted flag varieties for such groups correspond to varieties of flags of $\sigma$-isotropic subspaces of fixed dimension \cite{MPW1}. We will show in \S\ref{Tignol.sec} how to construct symplectic groups of any such prescribed $2$-index over suitable fields.


\subsection{Type $^{1\!}D_n$}
Over a base field $k$ of characteristic $\ne 2$, absolutely simple simply connected algebraic group of type $^{1\!\!}D_n$ are described as spinor groups $\Spin(A,\sigma)$ of a $2n$-degree central simple $k$-algebras $(A,\sigma)$ with orthogonal involution of trivial discriminant. The suitable generalization to include $k$ of characteristic $2$ is the notion of algebra with quadratic pair $(A, \sigma, f)$ as in \cite{KMRT}. We keep the same notations as before for the indexes of $A$, $(A, \sigma)$, and $(A,\sigma, f)$.\\

{\centering\noindent\makebox[360pt]{
  \includegraphics[width=13cm]{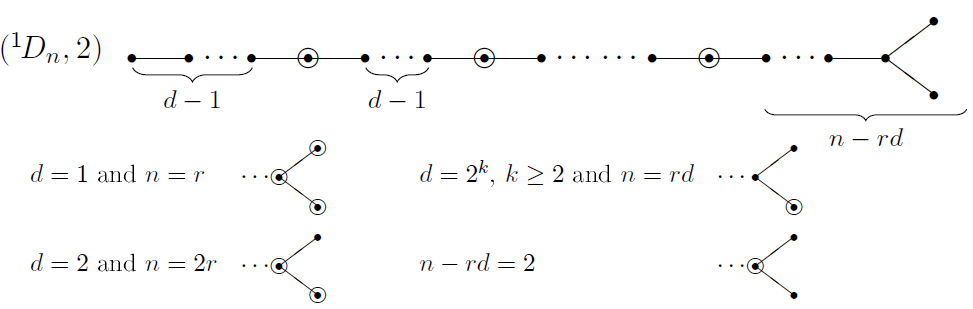}
}}

\emph{Distinguished orbits}: $\delta_0^2(G)=\{d,2d,...,rd\}$.\\

As detailed in \cite{MPW1} when $\car k \ne 2$, the associated twisted flag varieties are the varieties of flags of $\sigma$-isotropic subspaces of prescribed dimension.  This same description holds in all characteristics (replacing $(A, \sigma)$ with $(A, \sigma, f)$), as can be seen in case $A$ is split as in \cite[pp.~258--262]{Borel} and by Galois descent for general $A$, cf.~\cite[p.~219]{CG}.

Any such $2$-index can be realized as the index of the spinor group of a central simple $k$-algebra with orthogonal involution of trivial discriminant over a suitable field.


\subsection{Type $^{2\!}D_n$}

Absolutely simple simply connected algebraic groups of type $^{2\!}D_n$ are described by spinor groups of $2n$-degree central simple $k$-algebras endowed with an orthogonal involution of non-trivial discriminant (again, this notion is replaced by quadratic pairs to cover base fields of characteristic $2$). We denote here the discrete invariants associated to algebras with involution in the same way as for $^{1\!\!}D_n$, and the only torsion prime is $2$. As for the $^{1\!\!}D_n$ case, the twisted flag varieties for such groups are described in \cite{MPW1} and for any such prescribed $2$-index can be associated to a suitable spinor group.\\

{\centering\noindent\makebox[360pt]{
  \includegraphics[width=13cm]{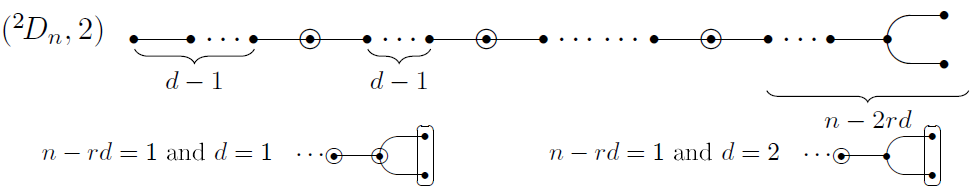}
}}

\emph{Distinguished orbits}: $\delta_0^2(G)=\{d,2d,...,rd\}$ if $rd<n-1$. As soon as $rd=n-1$, $\delta_0^2(G)=\{d,2d,...,(r-1)d,\{n-1,n\}\}$\\

\section{Tignol's construction} \label{Tignol.sec}

To complete the determination of all the values of the Tits $p$-indexes of classical groups, it remains to show that each of the previously-announced indices can be realized by suitable absolutely simple groups. Recall that a central simple algebra with involution $(A,\sigma)$ is adjoint to a (skew)-hermitian form $h_{\sigma}$ on a right $D$-module, where $D$ is a division algebra Brauer-equivalent to $A$. Adding hyperbolic planes to $h_{\sigma}$, the problem is reduced to the construction of \emph{anisotropic} central simple algebras with involutions $(A,\sigma)$ of any kind with $A$ of any index a power of $2$ over $2$-special fields. We reproduce here a construction of such algebras with involutions which is due to Jean-Pierre Tignol.

Let $\Gamma_n$ be a product of $n$ copies of $\mathbb{Z}_{(2)}$, the ring of rational numbers with odd denominators. For any field $K$, consider the field $K_n$ of power series $\sum_{\gamma\in \Gamma_n}a_{\gamma}x^{\gamma}$ whose support is well ordered with respect to the lexicographical order \cite{EnglerPrestel:valued}. The field $K_n$ is endowed with the valuation $v:K_n\longrightarrow \Gamma_n\cup \{\infty\}$ which sends an element to the least element of its support.

\begin{lem}
If $K$ is $2$-special, then $K_n$ is also $2$-special.
\end{lem}

\begin{proof}Let $L$ be a finite separable field extension of $K_n$. The valuation $v$ extends uniquely to $L$ and $K_n$ is maximally complete, hence the following equality holds \cite[Ch. 2]{schi:valuations}.
$$[L:K_n]=[\bar{L}:K]\cdot (v(L^{\times}):v(K_n^{\times}))$$
The field $K$ being assumed to be $2$-special, $[\bar{L}:K]$ is a power of $2$. Moreover the quotient group $v(L^{\times})/v(K_n^{\times})$ is torsion \cite[Theorem 3.2.4]{EnglerPrestel:valued} and $v(K_n^{\times})$ is $\Gamma_n$, hence the order of the quotient group $v(L^{\times})/v(K_n^{\times})$ is a power of $2$.
\end{proof}

We now describe Tignol's procedure to construct from any $K$-division algebra with involution $(D,\sigma)$ (of any kind) a family of anisotropic algebras with involutions of the same kind over $K_n$.

\begin{prop}\label{tignol}Let $M$ be a right $D_{K_n}$-module of rank $k$ which is at most $n$. The hermitian form
$$\begin{array}{ccccc}
h_k & : & M\times M & \longrightarrow & D_{K_n} \\
 & & (a_1,...,a_k,b_1,...b_k) & \mapsto & \sum_{i=1}^k\sigma_{K_n}(a_i)x^{\varepsilon_i}b_i \\
\end{array}$$
where $\varepsilon_i$ is the n-uple whose only non-zero entry is $1$ at the $i$-th position is anisotropic.
\end{prop}

\begin{proof}
Setting $v(d\otimes \lambda)=v(\lambda)$ for any $d\in D^{\times}$ and $\lambda\in K_n^{\times}$, the valuation $v$ extends to a $\sigma$-invariant valuation on $D_{K_n}$. One observe that for any element $a$ of $D_{K_n}^{\times}$, $v(\sigma_{K_n}(a)x^{\varepsilon_i}a)=\varepsilon_i+2v(a)$ belongs to $\varepsilon_i + 2\Gamma_n$ and thus
\[
v\left( \sum_{i=1}^k\sigma_{K_n}(a_i)x^{\varepsilon_i}a_i\right)=\min\{\varepsilon_i+2v(a_i), i=1,...,k\}.
\]
It follows that if $h_k(a_1,...,a_k,a_1,...,a_k)=0$, then $a_i=0$ for all $i=1,...,k$.
\end{proof}

\begin{cor}
Each of the previously described indices of type $(^{2\!}A_n,2)$, $(C_n,2)$, $(^{1\!\!}D_n,2)$, $(^{2\!}D_n,2)$ is the Tits $2$-index of a semisimple algebraic group defined on a suitable field.
\end{cor}

\begin{proof}
As seen in a previous discussion, it suffices to construct over $2$-special fields algebras with anisotropic involutions $(A,\sigma)$ of any kind, where the index of $A$ can be any power of $2$.

Take a sufficiently large transcendental field extension $K(x_1,...,x_s)$ of a field $K$, over which we can consider a division algebra with involution $(D,\sigma)$ of any kind. (For instance, $D$ may be chosen to be a tensor product of quaternion algebras.) Writing $L$ for the $2$-special closure of $K(x_1,...,x_s)$, we can apply Tignol's procedure to $(D_L,\sigma_L)$. Proposition \ref{tignol} gives rise to an anisotropic algebra with involution $(M_k(D_{L_n}),\sigma_{L_n})$ which is of the same kind as $(D,\sigma)$ and thus fulfills the required assumptions.
\end{proof}

\section{Tits $p$-indexes of exceptional groups}
In this section we will list the possible Tits $p$-indexes of \emph{exceptional} groups, meaning groups of the types omitted from \S\ref{classical.sec}.

\subsection{The cases $(G_2, 2)$, $(^3D_4, 3)$, $(F_4, 3)$, and $(E_8, 5)$} \label{ht1.ssec}

\setlength{\unitlength}{.5cm}

\newsavebox{\Esev}
\savebox{\Esev}(7,1.9){\begin{picture}(7,1.9)
    \multiput(1,0.7)(1,0){6}{\circle*{\darkrad}}
     \put(3,1.45){\circle*{\darkrad}}
    \put(1,.7){\line(1,0){5}}
    \put(3,1.45){\line(0,-1){0.75}}
\end{picture}}

\newsavebox{\Eviii}
\savebox{\Eviii}(8,1.9){\begin{picture}(8,1.9)
    \multiput(1,0.7)(1,0){7}{\circle*{\darkrad}}
     \put(3,1.45){\circle*{\darkrad}}
    \put(1,.7){\line(1,0){6}}
    \put(3,1.45){\line(0,-1){0.75}}
\end{picture}}

\newsavebox{\dEpic}
\savebox{\dEpic}(3.5,1){\begin{picture}(3.5, 1)
    \multiput(1.75,0.75)(0.75,0){2}{\circle*{\darkrad}}
    \multiput(1.75,0.25)(0.75,0){2}{\circle*{\darkrad}}
    \multiput(0.25,0.5)(0.75,0){2}{\circle*{\darkrad}}
    
    \put(0.25, 0.5){\line(1,0){0.75}}
    \put(1.75,0.75){\line(1,0){0.75}}
    \put(1.75, 0.25){\line(1,0){0.75}}
    
    \put(1.75,0.5){\oval(1.5,0.5)[l]}\end{picture}}
\newsavebox{\iEpic}
\savebox{\iEpic}(2.5,1){\begin{picture}(2.5,1)
    \multiput(0.25,0.25)(0.5,0){5}{\circle*{\darkrad}}
    \put(1.25,0.75){\circle*{\darkrad}}

    \put(0.25,0.25){\line(1,0){2}}
    \put(1.25,0.75){\line(0,-1){0.5}}
    
    \put(0.25,0.25){\circle{\lrad}}
    \put(2.25,0.25){\circle{\lrad}}
\end{picture}}

We now consider some groups $G$ relative to a prime $p$ in cases where $p$ does not divide the exponent of the center of the simply connected cover $\Gt$ and the Dynkin index $n_{\Gt}$ factors as $cp$ for some $c$ not divisible by $p$.  Definition \ref{am.def} then gives an element $b(G) \in H^3(k, \Z/p\Z(2))$ depending only on $G$.

\begin{prop} \label{ht1}
If the quasi-split type of $G$ and $p$ are one of $(G_2, 2)$, $(^3D_4, 3)$, $(F_4, 3)$, or $(E_8, 5)$, then the following are equivalent:
\begin{enumerate}
\item \label{ht1.split} $G$ is quasi-split by a finite separable extension of $k$ of degree not divisible by $p$.
\item \label{ht1.iso} $G$ is isotropic over a finite separable extension of $k$ of degree not divisible by $p$.
\item \label{ht1.zero} $b(G) = 0$.
\end{enumerate}
And for $G$ of type $G_2$, $F_4$, or $E_8$ the preceding are equivalent also to:
\begin{enumerate}
\setcounter{enumi}{3}
\item \label{ht1.motive} The Chow motive with $\F_p$ coefficients of the variety of Borel subgroups of $G$ is a sum of Tate motives.
\end{enumerate}
Moreover, for every field $k$, there exists a $p$-special field $F \supseteq k$ and an anisotropic $F$-group of the same quasi-split type as $G$.
\end{prop}

\begin{proof}
It is harmless to assume that $k$ is $p$-special.  Suppose \eqref{ht1.iso}, that $G$ is $k$-isotropic.  When we consult the list of possible Tits indexes from \cite{Ti:Cl}, we find that for $G$ of type $G_2$, $G$ is necessarily split.  If $G$ has type $^3D_4$, then the only other possibility is that the semisimple anisotropic kernel is isogenous to a transfer $R_{L/k}(A_1)$ where $L$ is cubic Galois over $k$.  But $L$ has no separable field extensions of degree 2, so a group of type $A_1$ is isotropic, hence $G$ is split.
If $G$ has type $F_4$, the only possibility has type $B_3$, which is isotropic, hence again \eqref{ht1.split}.  If $G$ has type $E_8$, the possibilities are $E_7$, $D_7$, $E_6$, $D_6$, or $D_4$, and all such groups are isotropic as in Table \ref{SG.table}.  Thus, \eqref{ht1.iso} implies \eqref{ht1.split}.

Assume now \eqref{ht1.zero}.  As $k$ is $p$-special (and because of our choice of $G$), $t_G = 0$, so we are reduced to showing that, for $\Gt^q$ the quasi-split inner form of the simply connected cover of $G$, the Rost invariant $r_{\Gt^q}$ has zero kernel.  By the main result of \cite{Gille:inv}, we may assume that $\car k = 0$.  The kernel is zero for type $G_2$ by \cite[p.~44]{GMS}, for $F_4$ it is \cite[\S40]{KMRT}, for $^3D_4$ it is \cite{G:rinvD4}, for $E_8$ it is \cite{Ch:mod5} or see \cite[15.5]{G:lens}.

Trivially, \eqref{ht1.split} implies the other conditions.  The remaining implication, that \eqref{ht1.motive} implies \eqref{ht1.split}, is \cite[Cor.~6.7]{PSZ:J}.

For existence, choosing a versal torsor under the simply connected cover of $G$ provides an extension $E \supseteq k$ and an  $E$-group $G'$ of the same quasi-split type as $G$ with $b(G') \ne 0$.  Then $G' \times F$ is the desired group, where $F$ is any  $p$-special closure of $E$.
\end{proof}

For $(G, p)$ as in the proposition, we needn't display the possible Tits $p$-indexes, because there are only two possibilities: quasi-split or anisotropic.

\begin{cor} \label{ht1.me}
Suppose $(G, p)$ is one of the pairs considered in Proposition \ref{ht1}, and $G'$ is a simple algebraic group that is an inner form of $G$.  Then $G$ and $G'$ are motivic equivalent mod $p$ if and only if $b(G)$ and $b(G')$ generate the same subgroup of $H^3(k, \Z/p\Z(2))$.
\end{cor}

\begin{proof}
The element $\nu_G \in H^1(k, G)$ represents a principal homogeneous space; write $K$ for its function field.  The kernel of $H^3(k, \Z/p\Z(2)) \to H^3(K, \Z/p\Z(2))$ is the group generated by $b(G)$ \cite[p.~129, Th.~9.10]{GMS}.  If the isomorphism of the Tits $p$-indexes extends to an isomorphism between $p$-indexes of $G_E$ and $G'_E$ for every extension $E$ of $k$, then $G_K$ and $G'_K$ are both quasi-split, hence $b(G')$ is in $\qform{b(G)}$; by symmetry the two subgroups generated by $b(G)$ and $b(G')$ are equal.  Conversely, if the two subgroups are equal, then for each extension $E$ of $k$, either $\res_{E/k}(b(G))$ is zero and both $G_E$ and $G'_E$ are quasi-split, or it is nonzero and both are anisotropic; in this case the isomorphism of the Tits $p$-indexes over $k$ clearly extends to an isomorphism over $E$.  Applying the main result of \cite{clercq:motivicequivalence} gives the claim.
\end{proof}

\subsubsection*{Remarks specific to $G_2$}
For $G, G'$ of type $G_2$, we have: \emph{$G \cong G'$ iff $b(G) = b(G')$,}  cf.~\cite[33.19]{KMRT} and \cite{Sp:ex}, i.e., motivic equivalence mod 2 is the same as isomorphism.  

The flag varieties for type $G_2$ are described in \cite[Example 9.2]{CG}.

\subsubsection*{Remarks specific to $F_4$ at $p = 3$}
For $G, G'$ of type $F_4$ over a 3-special field $k$, we have: \emph{$G \cong G'$ iff $b(G) = b(G')$.}  If $\car k = 0$, this is \cite{Rost:albert} and we can transfer this result to all characteristics using the same method as in \cite[\S9]{GPS}.


\subsection{Type $D_4$}
For $G$ of type $D_4$, $\Aut(\Delta)(\ksep)$ is the symmetric group on 3 letters, so $G$ has type $^tD_4$ with $t = 1$, 2, 3, or 6.  Groups of type $^1D_4$ or $^2D_4$ were treated in \S\ref{classical.sec}; this includes the case where $k$ is 2-special.   Thus it remains to consider groups of type $^3D_4$ and $p = 3$, which was treated in subsection \ref{ht1.ssec}.

\subsubsection*{Flag varieties}
For groups of type $^3D_4$ and $^6D_4$, the $k$-points of the twisted flag varieties are described in \cite{G:flag}.

\subsection{Type $F_4$ and $p = 2$}

Groups of type $F_4$, just as for type $G_2$, are all simply connected and adjoint and therefore all have Tits class zero; therefore the invariants $a$ and $b$ from Definition \ref{am.def} agree.  For $G$ a group of type $F_4$, one traditionally decomposes $b(G) \in H^3(k, \Z/6\Z(2))$ as $f_3(G) + g_3(G)$ for $f_3(G) \in H^3(k, \Z/2\Z(2))$ and $g_3(G) \in H^3(k, \Z/3\Z(2))$.  There is furthermore another cohomological invariant $f_5(G) \in H^5(k, \Z/2\Z(4))$, see \cite[37.16]{KMRT} or \cite[p.~50]{GMS} when $\car k \ne 2$ (in which case $f_5(G)$ belongs to $H^5(k, \Z/2\Z)$) or \cite[\S4]{Ptr:struct} for arbitrary $k$.  (These statements rely on viewing each group of type $F_4$ over $k$ as the automorphism group of a uniquely determined Albert $k$-algebra. For general background information on Albert algebras, see \cite[Ch.~IX]{KMRT}, \cite{Sp:ex}, or \cite{Ptr:ottawa}.)
Table \ref{F4.table} gives a dictionary relating the Tits index of $G$ with the values of these invariants; in the last column we give the signature of the Killing form for the Lie algebra over $\RR$ with that Tits index.  (Implicitly this is a statement of existence; one can calculate the signature of the Killing form from the Tits index by the formula from \cite[\S6]{Malagon}.)
\begin{table}[hbt]
\begin{tabular}{c|ccc|r} 
&&&&signature \\
Tits index of $G$ & $f_3(G)$&$f_5(G)$&$g_3(G)$&of real form \\ \hline
split&0&0&0&4 \\ 
\begin{picture}(4,1)
\multiput(0.5,0.5)(1,0){2}{\circle*{\darkrad}}
\multiput(2.5,0.5)(1,0){2}{\circle*{\darkrad}}
    \put(1.5,0.4){\line(1,0){1}}
    \put(1.5,0.6){\line(1,0){1}}
    \put(0.5,0.5){\line(1,0){1}}
    \put(2.5,0.5){\line(1,0){1}}
    
    \put(1.75, .35){\makebox(0.2,0.3)[s]{$<$}}
    
    \put(0.5,0.5){\circle{\lrad}}
    

\end{picture}&$\ne 0$&0&0&$-20$ \\
anisotropic&\multicolumn{3}{c|}{$f_5(G)$ and $g_3(G)$ not both zero}&$-52$
\end{tabular}
\caption{Tits index of a group of type $F_4$} \label{F4.table}
\end{table}

For type $F_4$, we should consider $p  = 2, 3$ by Table \ref{SG.table}.  The case $p = 3$ was handled in Proposition \ref{ht1}.  For $p = 2$, all three possible indexes occur over the 2-special field $\RR$, so they are also 2-indexes.  Alternatively, one can handle the $p = 2$ case  by noting that groups of type $F_4$ are of the form $\Aut(J)$ for an Albert algebra $J$ and then applying the Tits constructions of Albert algebras outlined in \cite{Ptr:struct}.

\begin{prop} \label{F4.me}
Groups $G$ and $G'$ be groups of type $F_4$ over a field $k$ are motivic equivalent mod 2 iff $f_3(G) = f_3(G')$ and $f_5(G) = f_5(G')$.  The groups $G$, $G'$ are motivic equivalent (mod every prime) iff $f_3(G) = f_3(G')$, $f_5(G) = f_5(G')$, and $g_3(G) = \pm g_3(G')$.
\end{prop}

\begin{proof}
The elements $f_d(G)$ for $d = 3, 5$ are symbols, so we may find $d$-Pfister quadratic forms $q_d$ whose Milnor invariant $e_d(q_d) \in H^d(k, \Z/2\Z(d-1))$ equals $f_d(G)$.  For $K_d$ the function field of $q_d$, the kernel of the map $H^d(k, \Z/2\Z(d-1)) \to H^d(K_d, \Z/2\Z(d-1))$ is generated by $f_d(G)$ as follows from \cite[Th.~2.1]{OVV} (if $\car k \ne 2$) as explained in \cite[p.~180]{EKM}.  The first claim now follows by the arguments used to prove Corollary \ref{ht1.me}.  The second claim follows from the first and Corollary \ref{ht1.me}.
\end{proof}

We thank Holger Petersson for contributing the following example.

\begin{eg}
Given a group $G$ of type $F_4$ over $k$, we write $G = \Aut(J)$ and furthermore write $J$ as a second Tits construction Albert algebra $J = J(B, \tau, u, \mu)$ for some quadratic \'etale $k$-algebra $K$, where $B$ is a central simple $K$-algebra of degree 3 with unitary $K/k$-involution $\tau$, $u \in B^\times$ is such that $\tau(u) = u$, and $\mu \in K^\times$ satisfies $N_{K/k}(\mu) = \mathrm{Nrd}_{B}(u)$.  Then $g_3(G) = -\mathrm{cor}_{K/k}([B] \cup [\mu])$ and $f_3(G)$ is the 3-Pfister quadratic form over $k$ corresponding to the unitary involution $\tau^{(u)} \!: x \mapsto u^{-1} \tau(x) u$ on $B$.  The 1-Pfister $N_{K/k}$ is a subform of the 3-Pfister $q$ corresponding to the involution $\tau$ \cite[Prop.~2.3]{PR:reduced} so there exist $\gamma_1, \gamma_2 \in k^\times$ such that $q = \pform{\gamma_1, \gamma_2} \ot N_{K/k}$.  Combining \cite[2.9]{PR:reduced} and \cite[7.9]{Ptr:ottawa} gives
\begin{equation} \label{f5.formula}
f_5(G) = \pform{\gamma_1, \gamma_2} \ot f_3(G).
\end{equation}

Define now $G' := \Aut(J')$ where $J'$ is the second Tits construction Albert algebra $J(B, \tau, u^{-1}, \mu^{-1})$.  Since the unitary involutions $\tau^{(u)}$ and $\tau^{(u^{-1})}$ of $B$ are isomorphic under the inner automorphism  $x \mapsto uxu^{-1}$, we have $f_3(G') = f_3(G)$.  As $\pform{\gamma_1, \gamma_2}$ does not change when passing from $J$ to $J'$, we find $f_5(G') = f_5(G)$.  As clearly $g_3(G') = -g_3(G)$, $G$ and $G'$ are motivic equivalent mod $p$ for all $p$.

It is unknown if $J'$ depends on the choice of expression of $J$ as a second Tits construction; perhaps it only depends on $J$. This is a specific illustration of the general open problem \cite[p.~465]{SeCGp}: Do the invariants $f_3$, $f_5$, and $g_3$ distinguish groups of type $F_4$?
\end{eg}

\subsubsection*{Flag varieties}
For groups of type $F_4$, the $k$-points of the twisted flag varieties are described in \cite[9.1]{CG}, relying on \cite{Racine:point} or \cite{Asch:E6}.  A portion of this description for $k = \RR$ can be found in \cite[28.22, 28.27]{Frd:E7.8}.  For $J$ an Albert algebra, $\Aut(J)$ is isotropic iff $J$ has nonzero nilpotents, and  $\Aut(J)$ is split iff $J$ is the split Albert algebra.

\subsection{Type $E_6$ and $p = 2, 3$}

For $ G$ of type $^1E_6$, the class $t_G$ has order dividing 3 and can be represented by a central simple algebra of degree 27 \cite[p.~213]{Ti:R} which we denote by $A$; it is only defined up to interchanging with its opposite algebra.
The list of possible Tits indexes from \cite{Ti:Cl} is reproduced in the first column of Table \ref{oE6.table}.  The constraints on the index of $A$ given in the second column can be deduced from the possible indexes of the Tits algebras of the semisimple anisotropic kernel as explained in \cite[p.~211]{Ti:R}.  In the column for 2-special fields, we give the signature of the Killing form on the real Lie algebra if one occurs with that Tits index.
\begin{table}[hbt]
\begin{tabular}{c|c|cc}
&&occurs as&occurs as\\
index of $G$&$\ind A$&a 2-index?&a 3-index? \\ \hline
split&1&yes (6)&yes \\
\begin{picture}(5,2)
    \multiput(0.5,0.5)(1,0){5}{\circle*{\darkrad}}
    \put(2.5,1.5){\circle*{\darkrad}}

    \put(0.5,0.5){\line(1,0){4}}
    \put(2.5,1.5){\line(0,-1){1}}

    \put(0.5,0.5){\circle{\lrad}}
    \put(4.5,0.5){\circle{\lrad}}
\end{picture}&1&yes ($-26$)&no \\
\begin{picture}(5,2)
    \multiput(0.5,0.5)(1,0){5}{\circle*{\darkrad}}
    \put(2.5,1.5){\circle*{\darkrad}}

    \put(0.5,0.5){\line(1,0){4}}
    \put(2.5,1.5){\line(0,-1){1}}

    \put(2.5,0.5){\circle{\lrad}}
    \put(2.5,1.5){\circle{\lrad}}
\end{picture} &  3&no&yes\\
anisotropic&divides 27&no&yes
\end{tabular}
\caption{Possible Tits indexes of groups of type $^1E_6$} \label{oE6.table}
\end{table}

\subsubsection{Type $^1E_6$ and $p = 3$}
Over a 3-special field, every group of type $^1D_4$ is split, therefore the Tits index with semisimple anisotropic kernel of that type (row 2 in Table \ref{oE6.table}) cannot occur.  The following table is justified in \cite[\S10]{GPS}, where the top row --- which is only proved assuming $\car k = 0$ --- refers to the mod-3 $J$-invariant defined in \cite{PSZ:J} describing the decomposition of the mod-3 Chow motive of the variety $X_\Delta$ of Borel subgroups of $G$.
\begin{center}
\begin{tabular}{c|ccccc}
$J_3(G)$&$(0,0)$&$(1,0)$&$(0,1)$&$(1,1)$&$(2,1)$ \\ \hline
\rb{Tits index of $G$}&\rb{split}&\begin{picture}(5,2)
    \multiput(0.5,0.5)(1,0){5}{\circle*{\darkrad}}
    \put(2.5,1.5){\circle*{\darkrad}}

    \put(0.5,0.5){\line(1,0){4}}
    \put(2.5,1.5){\line(0,-1){1}}

    \put(2.5,0.5){\circle{\lrad}}
    \put(2.5,1.5){\circle{\lrad}}
\end{picture}
&\multicolumn{3}{c}{\rb{$\cdots$ anisotropic $\cdots$}}\\
index of $A$&1&3&1&3&9 or 27 
\end{tabular}                           
\end{center}

\subsubsection{Type $^1E_6$ with $t_G = 0$} \label{E6.Tits}
For any group $G$ of type $^1E_6$ with $t_G = 0$, we get from Definition \ref{am.def} an element $a(G) \in H^3(k, \Z/6\Z(2))$, which we write as $f_3(G) + g_3(G)$ for $f_3(G)\in H^3(k, \Z/2\Z(2))$ and $g_3(G) \in H^3(k, \Z/3\Z(2))$.  It follows from \cite[11.1]{G:lens} that the simply connected cover of $G$ is the group of isometries of the cubic norm form of an Albert algebra $J$, and from general properties of the Rost invariant that $f_3(G)$ and $g_3(G)$ equal the corresponding values for the automorphism group $\Aut(J)$ of type $F_4$.  Combining the description of the flag varieties of $G$ in terms of subspaces of $J$ from \cite[\S7]{CG} as well as the relationships between values of the cohomological invariants and properties of $J$ from \cite[\S40]{KMRT} and \cite{Ptr:struct} gives the information in Table \ref{E6.triv}.
\begin{table}[hbt]
\begin{tabular}{c|cc} 
Tits index of $G$ & $f_3(G)$&$g_3(G)$\\ \hline
split&0&0 \\ 
\begin{picture}(5,2)
    \multiput(0.5,0.5)(1,0){5}{\circle*{\darkrad}}
    \put(2.5,1.5){\circle*{\darkrad}}

    \put(0.5,0.5){\line(1,0){4}}
    \put(2.5,1.5){\line(0,-1){1}}

    \put(0.5,0.5){\circle{\lrad}}
    \put(4.5,0.5){\circle{\lrad}}
\end{picture}
&$\ne 0$&0 \\
anisotropic&any&$\ne 0$
\end{tabular}
\caption{Table of possible Tits indexes for $G$ of type $^1E_6$ with $t_G = 0$} \label{E6.triv}
\end{table}

\begin{prop} \label{oE.me}
Let $G$ and $G'$ be groups of type $^1E_6$ over a field $k$ such that $t_G = t_{G'} = 0$.  Then $G$ and $G'$ are motivic equivalent modulo a prime $p$ if and only if the $p$-torsion components of $a(G)$ and $a(G')$ generate the same subgroup of $H^3(k, \Z/p\Z(2))$.
\end{prop}

\begin{proof}
Combine Table \ref{E6.triv} with the main result of \cite{clercq:motivicequivalence}.
\end{proof}

\subsubsection{Type $^1E_6$ and $p = 2$} \label{1E6.2}
Over a 2-special field, the Tits class of any group $G$ of type $E_6$ is zero so Table \ref{E6.triv} applies and in particular $G$ is isotropic.  

From this, we deduce the third column of Table \ref{oE6.table}.  Note that the two possible 2-indexes for a group $G$ of type $^1E_6$ are distinguished by the value of $f_3(G)$.

\begin{cor}
Groups $G$ and $G'$ of type $^1E_6$ over a field $k$ are motivic equivalent mod 2 if and only if $f_3(G) = f_3(G')$.
\end{cor}

\begin{proof}
As the Tits classes $t_G$ and $t_{G'}$ are zero over every 2-special field, the claim follows immediately from Proposition \ref{oE.me}.
\end{proof}

%

\subsubsection{Type $^2E_6$ and $p = 2$}
For $G$ of type $^2E_6$, the element $b(G)$ belongs to $H^3(k, \Z/4\Z(2))$.  (If $k$ is 2-special, then $t_G = 0$ and $a(G) = b(G)$.)  In this setting, the possible Tits 2-indexes have been determined in \cite[Prop.~2.3]{GPe}.  We reproduce that table here, as well as indicate the signature of the Killing form on the real simple Lie algebra with that Tits index, if such occurs.
\begin{table}[hbt]
\begin{tabular}{ccc} \\
index & $b(G) \in H^3(k, \Z/4\Z(2))$&occurs over $\RR$? \\ \hline
quasi-split & $0$&yes (2) \\
\begin{picture}(3.5, 1.1)
\put(0,0){\usebox{\dEpic}} 
\put(2.5,0.5){\oval(0.4,0.75)}
\put(0.25,0.5){\circle{\lrad}}
\end{picture}& nonzero symbol in $H^3(k, \Z/2\Z(2))$ killed by $K$&yes ($-14$)\\
\begin{picture}(3.5, 1.1)
\put(0,0){\usebox{\dEpic}}
\put(2.5,0.5){\oval(0.4,0.75)}
\end{picture}& symbol in $H^3(k, \Z/2\Z(2))$ not killed by $K$&no\\
\begin{picture}(3.5, 1)
\put(0,0){\usebox{\dEpic}} 
\put(0.25,0.5){\circle{\lrad}}
\end{picture}& in $H^3(k, \Z/2\Z(2))$, not a symbol&no \\
anisotropic & $\ne 0$&yes ($-78$)
\end{tabular}
\caption{Possible Tits 2-indexes for $G$ of type $^2E_6$} \label{dE6.table}
\end{table}


\subsubsection*{Flag varieties}
The $k$-points of the twisted flag varieties for groups of type $^1E_6$ are described in \cite[\S7]{CG} and for type $^2E_6$ in \cite[\S5]{GPe}.  
If $G$ is a group of type $^1E_6$ with $t_G = 0$, then the simply connected cover of $G$ is the 
the group of norm isometries of an Albert $k$-algebra $J$ and the variety of total flags has $k$-points $\{ S_1 \subset S_2 \subset S_3 \subset S_4 \subset S_5 \}$ where $S_i \subset J$ is a singular subspace, $\dim S_i = i$, and $S_5$ is not a maximal singular subspace.

\subsection{Type $E_7$ and $p = 2, 3$}
Let now $G$ have type $E_7$.  The class $t_G$ has order 1 or 2 and can be represented by a unique central simple algebra of degree 8 \cite[6.5.1]{Ti:R}, which we denote by $A$. Table \ref{E7.table} lists the possible Tits indexes for $G$.  As before, if a Tits index occurs over $\RR$, we list the signature of the Killing form in the 2-index column. 
\begin{table}[hbt]
\begin{tabular}{c|c|cc} 
&&occurs as&occurs as\\
index of $G$&$\ind A$&a 2-index?&a 3-index? \\ \hline
split&1&yes (7)&yes \\
\begin{picture}(7,1.9)
 \put(0,0){\usebox{\Esev}} 
 \put(1,0.7){\circle{\lrad}}
 \put(2,0.7){\circle{\lrad}}
 \put(3,0.7){\circle{\lrad}}
 \put(5,0.7){\circle{\lrad}}
 \end{picture}&2&yes ($-5$)&no\\
\begin{picture}(7,1.9)
 \put(0,0){\usebox{\Esev}} 
 \put(1,0.7){\circle{\lrad}}
  \put(5,0.7){\circle{\lrad}}
    \put(6,0.7){\circle{\lrad}}
\end{picture}&1&yes ($-25$)&no \\
\begin{picture}(7,1.9)
 \put(0,0){\usebox{\Esev}} 
 \put(1,0.7){\circle{\lrad}}
  \put(5,0.7){\circle{\lrad}}
\end{picture}&2&yes&no \\
\begin{picture}(7,1.9)
 \put(0,0){\usebox{\Esev}} 
  \put(5,0.7){\circle{\lrad}}
\end{picture}&2&yes&no\\
\begin{picture}(7,1.9)
  \put(0,0){\usebox{\Esev}} 
  \put(1,0.7){\circle{\lrad}}
\end{picture}&divides 4&yes&no \\
\begin{picture}(7,1.9)
  \put(0,0){\usebox{\Esev}} 
  \put(6,0.7){\circle{\lrad}}
\end{picture}&1&no&yes\\
anisotropic&divides 8&yes ($-133$)&no
\end{tabular}
\caption{Tits indexes of groups of type $E_7$} \label{E7.table}
\end{table} 

%
%

\subsubsection{Type $E_7$ and $p = 2$}  We must justify the third column of Table \ref{E7.table}.  As in \S\ref{1E6.2}, a group of type $^1E_6$ over a 2-special field is isotropic, so a group of type $E_7$ over a 2-special field cannot have semisimple anisotropic kernel of type $^1E_6$.

A versal form of $E_7$ over some field $F$ has Tits algebra of index 8 by \cite[p.~164]{MPW2}, \cite{M:max}, or \cite[Lemma 14.3]{GaGi}.  Going up to a 2-special extension of $F$, it is clear that there do exist groups of type $E_7$ that are anisotropic over a 2-special field and have Tits algebra of index 8.  The compact real form of $E_7$ has Tits algebra the quaternions (of index 2).  An example of an anisotropic group $G$ of type $E_7$ with $t_G = 0$ over a 2-special field is given in \cite[Example A.2]{G:lens}.

\subsubsection{Type $E_7$ and $p = 3$}
We now justify Table \ref{E7.3}, which implies the claims in the fourth column of Table \ref{E7.table}.  Let $G$ be a group of type $E_7$ over a 3-special field $k$, so $t_G = 0$.

We claim that $G$ is isotropic.  Suppose first the $\car k \ne 2, 3$.  Put $E_6$ and $E_7$ for the split simply connected groups of those types.  By \cite[12.13]{G:lens} the natural inclusion $E_6 \subset E_7$ gives a surjection in cohomology $H^1(k, E_6 \rtimes \mu_4) \to H^1(k, E_7)$.  As $H^1(k, \mu_4) = 0$, the class $\xi_G$ from Lemma \ref{tits.class} lies in the image of $H^1(k, E_6)$ and it follows that the Tits index of $(E_7)_{\xi_G}$, i.e., of the simply connected cover of $G$, is as in the bottom row of Table \ref{E7.3} or has more distinguished vertices. 
Note that as in Table \ref{E6.triv}, $(E_6)_{\xi_G}$ is split iff $g_3((E_6)_{\xi_G}) = 0$ iff $b(G) = 0$, and anisotropic iff $g_3((E_6)_{\xi_G}) \ne 0$ iff $b(G) \ne 0$.
This completes the justification of Table \ref{E7.3} if $\car k \ne 2, 3$.  If $k$ has characteristic 2 or 3, then arguing as in \cite[\S9]{GPS} reduces the claim to the case of characteristic zero.
\begin{table}[hbt]
\centering
\begin{tabular}{cc}
Tits 3-index of $G$&$b(G)$ \\ \hline
split & 0 \\
\begin{picture}(7,1.9)
  \put(0,0){\usebox{\Esev}} 
  \put(6,0.7){\circle{\lrad}}
\end{picture} & $\ne 0$
\end{tabular}
\caption{Possible Tits 3-indexes for a group $G$ of type $E_7$}
\label{E7.3}
\end{table}

\begin{prop}
Simple algebraic groups $G$ and $G'$ of type $E_7$ over a field $k$ are motivic equivalent mod 3 iff $b(G) = \pm b(G') \in H^3(k, \Z/3\Z(2))$.
\end{prop}

\begin{proof}
Combine Table \ref{E7.3} and the arguments used for Corollary \ref{ht1.me}.
\end{proof}

\subsubsection*{Flag varieties} The flag varieties for groups of type $E_7$ are described in \cite[\S4]{G:e7} and \cite{G:struct}.

\subsection{Type $E_8$ and $p = 2, 3$}

For type $E_8$, by Table \ref{SG.table} we should consider $p = 2, 3, 5$.  The case $p = 5$ was handled in Proposition \ref{ht1}.  We list the possible Tits indexes from \cite{Ti:Cl} in the first column of Table \ref{E8.table}.
\begin{table}[hbt]
\begin{tabular}{c|ccc} 
&occurs as&occurs as&occurs as\\
Tits index of $G$&a 2-index?&a 3-index?&a 5-index? \\ \hline
split&yes (8)&yes&yes \\
\begin{picture}(8,1.9)
 \put(0,0){\usebox{\Eviii}} 
 \put(1,0.7){\circle{\lrad}}
 \put(5,0.7){\circle{\lrad}}
 \put(6,0.7){\circle{\lrad}}
 \put(7,0.7){\circle{\lrad}}
 \end{picture}&yes ($-24$)&no&no\\
\begin{picture}(8,1.9)
 \put(0,0){\usebox{\Eviii}}
 \put(1,0.7){\circle{\lrad}}
 \put(7,0.7){\circle{\lrad}}
 \end{picture}&yes&no&no\\
\begin{picture}(8,1.9)
 \put(0,0){\usebox{\Eviii}} 
 \put(6,0.7){\circle{\lrad}}
 \put(7,0.7){\circle{\lrad}}
 \end{picture}&no&yes&no\\
\begin{picture}(8,1.9)
 \put(0,0){\usebox{\Eviii}} 
 \put(1,0.7){\circle{\lrad}}
 \end{picture}&yes&no&no\\
\begin{picture}(8,1.9)
 \put(0,0){\usebox{\Eviii}} 
 \put(7,0.7){\circle{\lrad}}
 \end{picture}&yes&no&no\\
anisotropic&yes ($-248$)&yes&yes
\end{tabular}
\caption{Possible Tits indexes for a group of type $E_8$} \label{E8.table}
\end{table}

\subsubsection{Type $E_8$ and $p = 2$}  As anisotropic strongly inner groups of types $D_4$, $D_6$, $D_7$, and $E_7$ exist over 2-special fields by the previous sections, and a compact $E_8$ exists over $\RR$, the ``yes'' entries in the second column of Table \ref{E8.table} are clear.  For the one ``no'', we refer to Table \ref{oE6.table}.

\subsubsection{Type $E_8$ and $p = 3$} In view of results for groups of smaller rank, it suffices to justify the existence of an anisotropic $E_8$ over a 3-special field.  For this, we refer to the following table, which is justified in \cite[\S10]{GPS}.
\[
\begin{tabular}{c|ccc}
$J_3(G)$&$(0,0)$&$(1,0)$&$(1,1)$\\ \hline
\rb{Tits 3-index of $G$}&\rb{split}&
\begin{picture}(8,1.9)
    \multiput(1,0.7)(1,0){7}{\circle*{\darkrad}}
    \multiput(6,0.7)(1,0){2}{\circle{\lrad}}
    \put(3,1.45){\circle*{\darkrad}}

    \put(1,.7){\line(1,0){6}}
    \put(3,1.45){\line(0,-1){0.75}}
\end{picture}
&\rb{anisotropic}\\
$r_{G_0}(\xi)$&$0$&nonzero symbol&otherwise
\end{tabular}
\]

\subsubsection*{Flag varieties} Currently there is no concrete description of the flag varieties of $E_8$ available in the form analogous to the others presented here.  However, groups of type $E_8$ can be viewed as the automorphism group of various algebraic structures as explained in \cite{GG:simple} (such as a 3875-dimensional algebra, for fields of characteristic $\ne 2$), so the methods of \cite{CG} can in principle be used to give a concrete description of the flag varieties in terms of such structures.

\bibliographystyle{amsplain}
\bibliography{skip_master}

\end{document}